\documentclass[12pt]{amsart}
\usepackage{graphicx}
\usepackage{amsmath}
\usepackage{amsfonts}
\usepackage{amssymb}
\usepackage[latin1]{inputenc}
\usepackage{multirow}
\usepackage{xcolor}
\usepackage{subcaption}
\newcommand{\f}{\mathrm{Fix}}

\newtheorem{theorem}{Theorem}[section]
\newtheorem{definition}[theorem]{Definition}
\newtheorem{lemma}[theorem]{Lemma}

\newtheorem{proposition}[theorem]{Proposition}
\newtheorem{remark}[theorem]{Remark}
\newtheorem{example}[theorem]{Example}
\newtheorem{observation}[theorem]{Observation}

\def \ord {\operatorname{ord}}

\begin{document}

\title[On the functional graph of $f(X)=X(X^{q-1}-c)^{q+1}$]{On the functional graph of $f(X)=X(X^{q-1}-c)^{q+1},$ over quadratic extensions of finite fields}


\author{Josimar J.R. Aguirre, Ab\'ilio Lemos and Victor G.L. Neumann}

\maketitle


\vspace{8ex}

\begin{abstract}
Let $\mathbb{F}_{q}$ be the finite field with $q$ elements. 
In this paper we will describe the dynamics of the map $f(X)=X(X^{q-1}-c)^{q+1},$ with $c\in\mathbb{F}_{q}^{\ast},$ over the finite field $\mathbb{F}_{q^2}$.
\end{abstract}

\section{Introduction and Notations}

For $q$ a prime power, denote $\mathbb{F}_{q}$ the finite field with $q$ elements, $\mathbb{F}_{q}^{*}=\mathbb{F}_{q} \backslash\{0\}$, and $\mathbb{F}_{q}[x]$ the ring of polynomials over $\mathbb{F}_{q}$. For $\nu\in\mathbb{N},$ let $\mathbb{F}_{q^\nu}$ be the extension of $\mathbb{F}_{q}$ of degree $\nu.$ We define the {\it functional graph} $f:\mathbb{F}_{q^\nu}\longrightarrow\mathbb{F}_{q^\nu}$ as the directed graph $\mathcal{G}(f)=(\mathcal{V},\mathcal{E})$ where $\mathcal{V}=\mathbb{F}_{q^\nu}$ and $\mathcal{E}=\{\langle x,f(x) \rangle \mid x \in \mathbb{F}_{q^\nu} \}$. We define the {\it iterations} of $f$ as $f^{(0)}(x)=x$ and $f^{(n+1)}(x)=f(f^{(n)}(x))$.
Fixing $\alpha \in \mathbb{F}_{q^\nu}$, there are minimal integers $0 \leq i<j$ such that $f^{(i)}(\alpha)=f^{(j)}(\alpha)$.
In the case when $i>0$, we call the list $\alpha, f(\alpha), f^{(2)}(\alpha), \ldots, f^{(i-1)}(\alpha)$ a
{\it pre-cycle} and $f^{(i)}(\alpha), f^{(i+1)}(\alpha), \ldots, f^{(j-1)}(\alpha)$ a {\it cycle} of length $(j-i)$  or
a  $(j-i)$-cycle. If $\alpha$ is an element of a cycle, we call it a {\it periodic element} and, if $f(\alpha)=\alpha$ we say it is a {\it fixed point}. For each element $\alpha \in \mathbb{F}_{q^\nu}$, we define $f^{-1}(\alpha)=\{\gamma \in \mathbb{F}_{q^\nu} \mid f(\gamma)= \alpha \}$ and  $\f(f)=\{\alpha\in\mathbb{F}_{q^\nu} \mid f(\alpha)=\alpha\}$ as the {\it set of pre-images} of $\alpha$ and the {\it set of fixed points} of $f,$ respectively. We will  introduce some graph notation as in \cite{AB,BT, QP,QR,QR2}. Let $m$ and $n$ be positive integers.
\begin{enumerate}
	\item $\mathcal{C}_n$ denotes an {\it oriented cycle} of length $n.$
	\item $\mathcal{T}_m$ denotes the {\it directed tree} composed by $m+1$ points, $P_1,P_2,\ldots P_{m+1}$, where $P_i$ is directed to $P_{m+1}$
	for all $1 \le i \le m.$
	\item $(\mathcal{C}_n,\mathcal{T}_m)$ denotes the graph obtained after replacing each point of
	$\mathcal{C}_n$ by a tree
	isomorphic to $\mathcal{T}_m$. 
	\item $ \mathcal{G}\oplus  \mathcal{H}$ denotes the disjoint union of the graphs $ \mathcal{G}$ and $ \mathcal{H}$.
	\item $\mathcal{G} = k \times \mathcal{H}$ means that $\mathcal{G} = \bigoplus_{i=1}^k \mathcal{H}_i$, where each $\mathcal{H}_i$ is isomorphic to $\mathcal{H}$.
\end{enumerate}

With regard to the study of maps of type $f:\mathbb{F}_{q^\nu}\longrightarrow\mathbb{F}_{q^\nu}$, we are interested in determining

\begin{enumerate}
	\item The number of fixed points.
	\item The number of cycles and its respective lengths.
	\item The number of connected components.
\end{enumerate}

The dynamic of iterations of polynomials over finite fields have attracted much attention in recent years, in part due to their applications in cryptography and integer factorization methods like \textit{Pollard rho algorithm} (see \cite{pollard}, \cite{pollard2}). The study of this kind of dynamics was initiated by Vivaldi \cite{vivaldi} and continued throughout the years, studying properties of {\it functional graph} from different important applications in finite fields.
In \cite{PR}, L. Reis and D. Panario have described the functional graph associated to the $\mathbb{F}_q$-linearized irreducible divisors of $x^n-1$. In \cite{chow}, Chow and Shparlinski obtained results on repeated exponentiation modulo a prime. In \cite{QP}, \cite{QP2}, Chebyshev polynomials and Redei functions were studied respectively.
Some rational maps have also been treated, such as $x + x^{-1}$ for \textit{char}$(\mathbb{F}_q) = \{2,3,5\}$ in \cite{ugolini} and their generalizations in \cite{ugolini2}. In \cite{ugolini3}, Ugolini studied rational maps induced by endomorphisms of ordinary elliptic curves.
Recently, in \cite{AB}, Alves J. and Brochero F. studied the digraph associated to the map $x \to g(x) := x^n h\left(x^{\frac{q-1}{m}}\right)$ where $h(x) \in \mathbb{F}_q[x]$, $n>1$ and $m$ is a positive integer determined by $g(x)$.
Also, in \cite{BT}, Brochero F. and Teixeira H. R. described completely the dynamics of the map $f(X)=c(X^{q+1}+aX^2)$, for $a=\{ \pm 1\}$ and $c \in \mathbb{F}_q^*$, over the finite field $\mathbb{F}_{q^2}$.



A polynomial $f \in \mathbb{F}_{q^\nu}[x]$ is called a {\it permutation polynomial} (PP) of $\mathbb{F}_{q^\nu}$ if it induces a bijection from $\mathbb{F}_{q^\nu}$ to itself. It is clear that if a map is defined by a PP, then its functional graph is composed only by cycles. But, what happens to the functional graph if we make a perturbation in such a PP or in the field over which it is defined? For example, the polynomial $f(X)=X(X^{q-1}-c)^{q+1},$ is a PP of $\mathbb{F}_{q^\nu}$ if and only if $c\notin\left(\mathbb{F}_{q^\nu}^\ast\right)^{q-1},$ i.e., $c$ is not a $(q-1)$-th power of an element of $\mathbb{F}_{q^\nu}^\ast$ (see Exercise 7.5 of \cite{LN}). Notice that if we consider the polynomial $f(X)=X(X^{q-1}-c)^{q+1}$ over $\mathbb{F}_{q^2}$ it can be seen as a perturbation of the identity, since for $c=0$ we have $f(X)=X^{q^2}=X.$ In the present paper, for $c \in \mathbb{F}_q^\ast$, we study the dynamics of the map $f:\mathbb{F}_{q^2}\longrightarrow\mathbb{F}_{q^2}$, where $f(X)=X(X^{q-1}-c)^{q+1}$ in order to present its functional graph $\mathcal{G}(f)$. 

Note that $c\in\left(\mathbb{F}_{q^2}^\ast\right)^{q-1} \cap \ \mathbb{F}_q^*$ is equivalent to the condition $c=\pm 1$, thus the map $f(X)=X(X^{q-1}-c)^{q+1},$ is a PP of $\mathbb{F}_{q^2}$ with $c \in \mathbb{F}_q^*$ if and only if $c\neq\pm 1$.

\section{The functional graph of $f(X)=X(X^{q-1}-c)^{q+1}$ in odd characteristic}\label{sectiongen}

In this section, we will assume that $q$ is an odd prime power.

\begin{definition}
	For every $\alpha \in \mathbb{F}_q$, let $\chi_2$ be the quadratic character over $\mathbb{F}_q$ defined as
	$$
	\chi_2(\alpha) =
	\left\{ \begin{array}{cl}
		1 & \text{if } \alpha \text{ is a square in } \mathbb{F}_q , \\
		-1 & \text{if } \alpha \text{ is not a square in } \mathbb{F}_q ,\\
		0 & \text{if } \alpha=0.
	\end{array}
	\right.
	$$
\end{definition}

Let $\{1,\beta \}$ be a basis of $\mathbb{F}_{q^2}$ over $\mathbb{F}_q$,
where $\beta^2=a \in \mathbb{F}_q^*$ and $a$ is not a square in $\mathbb{F}_q$.
Thus, every element $\alpha \in \mathbb{F}_{q^2}$
can be written as $\alpha=x+y\beta$, with $x,y\in \mathbb{F}_q$.

\begin{lemma}\label{lema1}
Let $\alpha=x+y\beta \in \mathbb{F}_{q^2}$, with $x,y \in \mathbb{F}_q$. We have
$f(\alpha)=g(x,y)  \alpha$,
where
\[
g(x,y)= \frac{(1-c)^2x^2-(1+c)^2y^2a}{x^2-y^2a}\in \mathbb{F}_q.
\]
\end{lemma}
\begin{proof}
Observe that $(x+y\beta)^q=x-y\beta$ since $\beta^{q-1}=a^{\frac{q-1}{2}}=-1$. Therefore,
\begin{align*}
f(x+y\beta) & = (x+y\beta)\left( (x+y\beta)^{q(q-1)}-c \right) \left( (x+y\beta)^{q-1}-c \right) \\
          & = \left( (x-y\beta)^{q-1}-c \right) \left( (x-y\beta)-c(x+y\beta) \right) \\
          & = \left( \frac{((1-c)x+(1+c)y\beta)}{x-y\beta} \right) ((1-c)x-(1+c)y\beta) \\
					& = \frac{(1-c)^2x^2-(1+c)^2y^2a}{x-y\beta} = g(x,y) (x+y\beta).
\end{align*}
\end{proof}

We use this result to find the cardinality of $\f(f)$.

\begin{theorem}\label{fixed}
We have
\[
|\f (f)|=
\left\{\begin{array}{ll}
  q &  \textrm{ if } c^2=4 ,\\	 
  2q-1 & \textrm{ if } \chi_2\left(\frac{c+2}{c-2}\right)=-1, \\
	1   & \textrm{ if } \chi_2\left(\frac{c+2}{c-2}\right)=1.
\end{array}\right.
\]

\end{theorem}

\begin{proof}
If $\alpha=0$, then $ f(\alpha)=\alpha(\alpha^{q-1}-c)^{q+1}=0$. 
Let $\alpha =x+y\beta \neq 0$.  From the previous lemma we have
\begin{align*}
f(\alpha)=\alpha & \Longleftrightarrow \left( \frac{(1-c)^2x^2-(1+c)^2y^2a}{x^2-y^2a} \right) = 1  \\	 
	   & \Longleftrightarrow (c-2)x^2 = (c+2)y^2 a. 
\end{align*}
If $c=2$, then  $y=0$ and  $\alpha=x$, and if $c=-2$, then $x=0$ and $\alpha=y\beta$. In both cases, we have $|\f(f)|=q$. 
Now, suppose $c^2\neq 4$.  In this case, $f(\alpha)=\alpha$ is equivalent to
$$
x^2= \dfrac{a(c+2)}{c-2} y^2.
$$
We have two possibilities to consider.\\
{\bf Case 1}: $\chi_2\left(\frac{c+2}{c-2}\right)=-1.$ \\
In this case, there exists $t \in \mathbb{F}_q^*$ such that $t^2=\dfrac{a(c+2)}{c-2}$. Therefore, the solutions of $f(\alpha)=\alpha$ are of the form $\alpha=x+y\beta$ with $x=\pm ty$, and we have 
a total of $2(q-1)+1=2q-1$ fixed points, including $\alpha=0$.\\
{\bf Case 2}: $\chi_2\left(\frac{c+2}{c-2}\right)=1.$ \\
In this case, there are no fixed points except for $\alpha = 0$.
\end{proof}

The following result is a consequence of Lemma \ref{lema1}.

\begin{theorem}\label{ccz}
We have:
\begin{enumerate}
\item[(a)] The connected component that contains zero is composed only by zero directed to itself if $c\neq \pm1$.
\item[(b)] The connected component that contains zero is composed by zero directed to itself and the $q-1$ elements
of $\mathbb{F}_q^*$ directed to zero if $c=1$, or the $q-1$ elements
of $\beta\mathbb{F}_q^*$ directed to zero if
$c=-1$.  
\end{enumerate}
\end{theorem}

\begin{proof}
If $\alpha \neq 0$, from Lemma \ref{lema1} we have
\begin{align*}
f(\alpha)= 0 & \Longleftrightarrow (1-c)^2x^2-(1+c)^2y^2a =0 \\	 
	   & \Longleftrightarrow x^2(c-1)^2=ay^2(c+1)^2. 
\end{align*}
%
%
%

If $c\notin \{ 1,-1\},$ $x\neq 0$ and $y\neq 0$, then $\frac{x^2}{y^2}=a\left(\frac{c+1}{c-1}\right)^2$
has no solution, since $a$ is not a square in $\mathbb{F}_q.$ It follows that $f^{-1}(0)=\{0\}.$
If $c=1,$ then $x\in\mathbb{F}_q$ and $y=0,$ and if $c=-1,$ then $y\in\mathbb{F}_q$ and $x=0$. In both cases we have $|f^{-1}(0)|=q.$ 

Now, for $x\neq0,$ we need to prove that $f^{-1}(x+0\beta)=\emptyset$ (for the case $c=1$) and, for $y\neq0,$ $f^{-1}(0+y\beta)=\emptyset$ (for the case $c=-1$). For $c=1,$ suppose  $z+w\beta \in f^{-1}(x+0\beta)$. Thus, according to Lemma \ref{lema1}
\[
\frac{-4w^2a}{z^2-w^2a}(z+w\beta)=x+0\beta \Longleftrightarrow\ w=0 \textrm{ and } \frac{-4w^2za}{z^2-w^2a}=x \Longrightarrow x=0,
\] 
a contradiction. Therefore, $f^{-1}(x+0\beta)=\emptyset.$  Similarly, a contradiction is obtained for the case $c=-1.$
\end{proof}

In all examples of this paper we identify $\mathbb{F}_{q^2}$ with $\mathbb{F}_q^2$
via the relation $x+y\beta \longmapsto (x,y)$.

\begin{example}
Let $q=7$, $c =1$ and $\beta \in \mathbb{F}_{7^2}$ such that $\beta^2=3$.
The connected component that contains $0\in\mathbb{F}_{7^2}$ (represented by $(0,0)$ in the graph) is shown in Figure \ref{fig1}.
 \end{example}

\begin{example}
Let $q=7$, $c =-1$ and $\beta \in \mathbb{F}_{7^2}$ such that $\beta^2=3$.
Then the connected component that contains  $0 \in\mathbb{F}_{7^2}$ is shown in Figure \ref{fig2}.
\end{example}

\begin{figure}[h!]
\begin{subfigure}[h]{0.4\linewidth}
\includegraphics[scale=0.6]{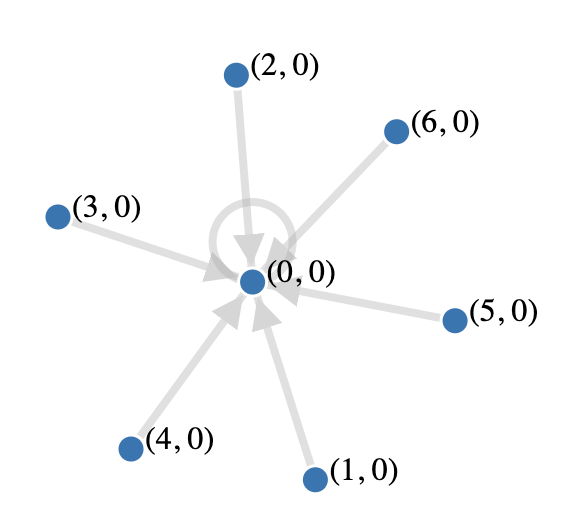}
\caption{Component of zero for $c=1$}\label{fig1}
\end{subfigure}
\hfill
\begin{subfigure}[h]{0.4\linewidth}
	\includegraphics[scale=0.6]{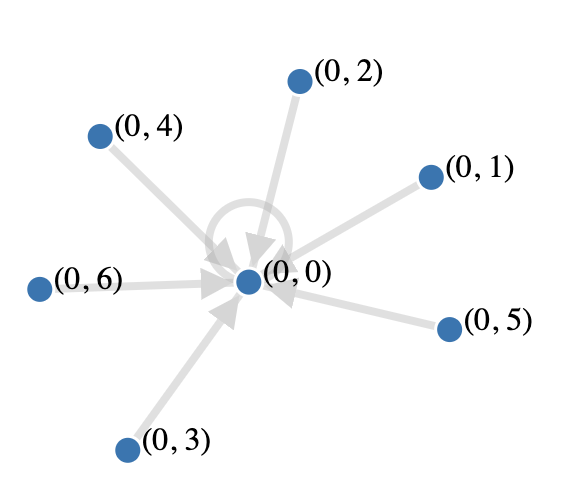}
	\caption{Component of zero for $c=-1$}\label{fig2}
\end{subfigure}
\label{fig:boat1}
\end{figure}

Note that the functional graphs for $c= \pm 1$ in the last example are equivalents, considering the map $(x,y) \to (y,x)$. This happens in the general case as we can see in the following observation. 

\begin{observation}
If we define $f_c(x+y\beta):=f(x+y \beta)$ as a function of $c$, then we have $f_c(x+y \beta)=g(y,x)(x+y\beta)$ (from Lemma \ref{lema1}) and $f_{-c}(y+x \beta^{-1})=g(x,y) (y+x \beta^{-1})$. Thus, there is an isomorphism between the \textit{functional graphs} related to the non-zero values $c$ and $-c$.  
\end{observation}

%
%

From Lemma \ref{lema1}, we can see that the graph $\mathcal{G}(f)$ can be partitioned into smaller graphs along each
``line'' passing through the origin of $\mathbb{F}_{q^2}$ when, as mentioned before, we identify $\mathbb{F}_{q^2}$ with $\mathbb{F}_{q}^2$ via the relation $x+y\beta \mapsto (x,y)$. Thus, we denote by $\mathbb{P}^1$ the space of lines in $\mathbb{F}_{q^2}$ passing through the origin, that is, each element of $\mathbb{P}^1$ is of the form
\[
L_{[u:v]} :=
 \{  \lambda (u+ v  \beta) \in \mathbb{F}_{q^2} \mid \lambda \in \mathbb{F}_q   \}. 
\]
The definitions of $\mathbb{P}^1$ and $L_{[u:v]}$ that we just made in odd characteristic will also remain valid for even characteristic.

Observe that if we exclude $0\in \mathbb{F}_{q^2}$, we find that $\mathbb{F}_{q^2}^*$ is partitioned into $q+1$ lines of the form $L_{[u:v]} \backslash \{0\}$.
We will prove that each of these lines is divided into cycles of the same size, except perhaps for one of them 
(see Theorem \ref{ccz}).

\begin{theorem}\label{teo1-v2}
For $L_{[u:v]}  \in \mathbb{P}^1$ we have $g(u,v)=0$  if and only if $c=1$ and $v=0$, or
$c=-1$ and $u=0$. We also have that 
for each $L_{[u:v]}  \in \mathbb{P}^1$ such that $g(u,v)\neq 0$, the set
$L_{[u:v]}\backslash \{ 0\}$ forms $\frac{q-1}{d}$ cycles of length $d$ in the graph $\mathcal{G}(f)$,
where $d$ is the multiplicative order of $g(u,v)$ in $\mathbb{F}_q^*$.
\end{theorem}
\begin{proof}
The  first part of the statement follows from
Theorem \ref{ccz}.
	From Lemma \ref{lema1}, we have $f(x+y\beta)=g(x,y)(x+y\beta)$. Note that $g(\lambda x, \lambda y)=g(x,y)$ for all 
	$\lambda \in \mathbb{F}_q^*$, thus $g(x,y)$ is constant in each set $L_{[u:v]} \backslash \{ 0 \}$ and
\begin{equation}\label{cycle1}
	f^{(n)}(x+y\beta)= g(u,v)^n (x+y\beta), 
\end{equation}
for all $n \ge 0$ and
all $x+y\beta \in L_{[u:v]} \backslash \{ 0 \}$. Equation \eqref{cycle1} implies that
if $g(u,v) \neq 0$, then each 
set $L_{[u:v]} \backslash \{ 0 \}$ is partitioned into $\frac{q-1}{d}$ cycles of length $d$
 in the graph $\mathcal{G}(f)$.
\end{proof}

\begin{remark}\label{remAB}
The polynomial studied in this paper can be obtained from the map $x \to x^n h\left(x^{\frac{q^2-1}{m}}\right)$ in \cite{AB} considering $n=1,$ $h(x)=(x-c)^{q+1}$ and $m=q+1$. For example,  Theorem \ref{ccz} can be seen as a particular case of  \cite[Theorem 2.4]{AB}.
Also, the authors presented results about the connected components that do not contain zero  in \cite[Theorem 2.7]{AB},
but we observe that this result is not applicable when $n=1.$ Therefore,
Theorem \ref{principalodd} is not a consequence of  \cite[Theorem 2.7]{AB}.
This remark is independent of the parity of $q$.
\end{remark}

\begin{example}\label{example1}
Let $q=11$, $c =1$, and $\beta \in \mathbb{F}_{11^2}$ such that $\beta^2=7$.
The functional graph of $f(X)=X(X^{10}-1)^{12}$ is 
shown in Figure \ref{fig5}.

 In Table \ref{table1} we show the number of cycles of the graph
	$\mathcal{G}(f)$ in each set $L_{[u:v]} \backslash \{ 0\}$, where $L_{[u:v]} \neq L_{[1:0]}$.
{\small
	\begin{table}[h]
		\centering
		\begin{tabular}{ccc}
			$[u:v]$ & $d={\ord }( g(u,v) )$ & Number of cycles of length $d$ \\
			\hline
            $[1: 1],[10: 1]$ & $1$   & $10+10=20$ \\
            $[0:1],[4: 1],[7: 1]$ &   $5$    & $2+2+2=6$ \\
            $[2: 1],[3: 1],[5: 1],[6: 1],[8: 1],[9: 1]$ & $10$ & $1+1+1+1+1+1=6$
		\end{tabular}
		\caption{Cycles over $\mathbb{F}_{11^2}$}
		\label{table1}
	\end{table}
}
\end{example}

The following result summarizes Theorem \ref{ccz} and Theorem \ref{teo1-v2}.
Observe that if $c\neq \pm 1$,
the connected component of zero 
	is isomorphic to the initial term
	$\mathcal{C}_1$, and if $c=1$ or $c=-1$,
	the connected component of zero 
	is isomorphic to the initial term
	$(\mathcal{C}_1,\mathcal{T}_{q-1})$.

\begin{theorem}\label{principalodd}
Let $q$ be a power of an odd prime. The functional graph of the function
$f(X)=X(X^{q-1}-c)^{q+1}$ over $\mathbb{F}_{q^2}$, with $c \in \mathbb{F}_q\backslash \{0,1,-1 \}$, is isomorphic to
\[
\mathcal{G} = \mathcal{C}_1 \oplus \bigoplus_{L_{[u:v]}\in \mathbb{P}^1} \frac{q-1}{{\ord} (g(u,v))} \times  \mathcal{C}_{{\ord}(g(u,v))}.
\]
If $c=\pm 1$, the functional graph is isomorphic to
\[
\mathcal{G} = (\mathcal{C}_1,\mathcal{T}_{q-1})  \oplus 
\bigoplus_{L_{[u:v]}\in \mathbb{P}^1 \backslash \{L_0 \}} 
\frac{q-1}{{\ord} (g(u,v))} \times \mathcal{C}_{{\ord}(g(u,v))},
\]
where $L_0=L_{[1:0]}$ if $c=1$, and $L_0=L_{[0:1]}$ if $c=-1$.
\end{theorem}

\begin{figure}[h!]
	\includegraphics[scale=0.9]{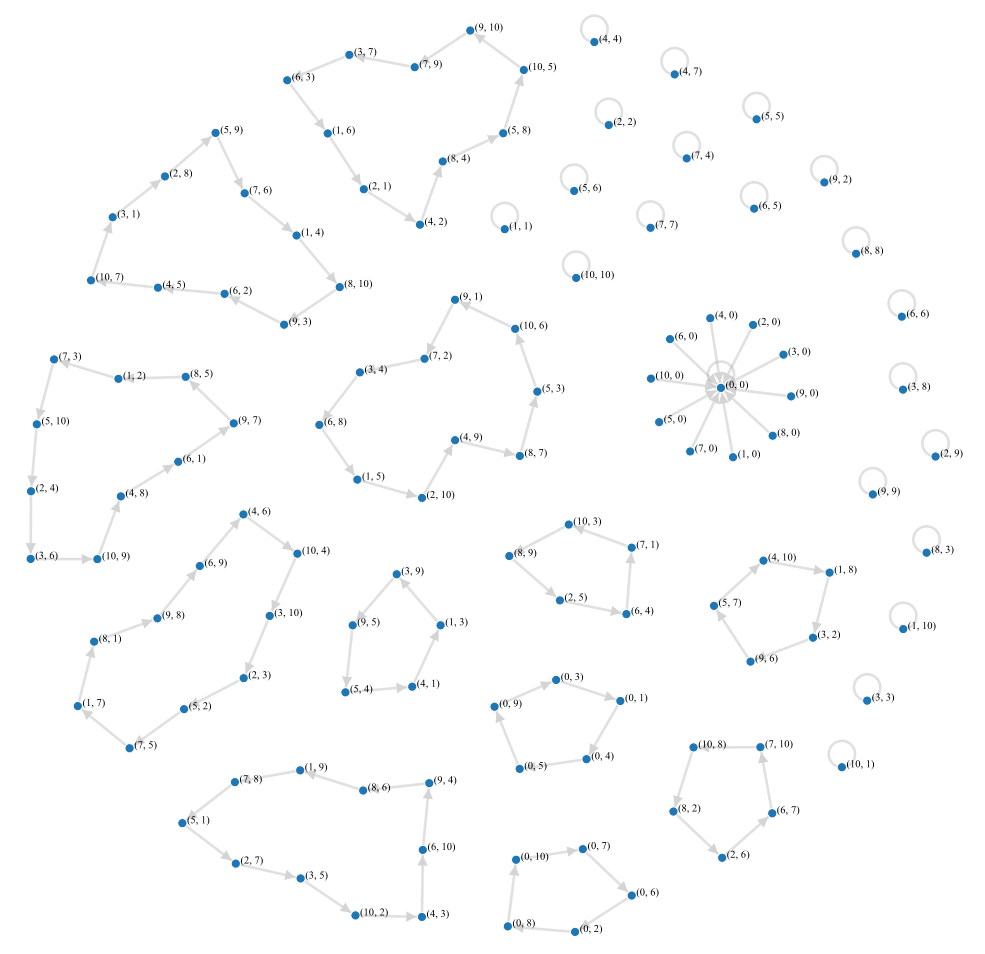}
	\caption{The functional graph of $f(X)=X(X^{10}-1)^{12}$ over $\mathbb{F}_{11^2}$}\label{fig5}
\end{figure}

\section{The functional graph of $f(X)=X(X^{q-1}-c)^{q+1}$
	in even characteristic}\label{sectioneven}

In this section, we will assume that $q$ is an even prime power. 
From \cite[Proposition 5.9]{livro}, 
for a given $a \in \mathbb{F}_q$,
the polynomial $h(x)=x^2+x+a \in \mathbb{F}_q[x]$  is irreducible
if and only if $\operatorname{Tr}(a) =1$, where the trace function $\operatorname{Tr}$ is defined from $\mathbb{F}_q$ to the prime field $\mathbb{F}_2$.
Thus, in this section, we will also assume that $a \in \mathbb{F}_q$ with  $\operatorname{Tr}(a) =1$,
and $\beta \in \mathbb{F}_{q^2}$ with $\beta^2+\beta +a=0$. As in the previous section,
the set $\{1,\beta\}$ forms a basis for $\mathbb{F}_{q^2}$ over $\mathbb{F}_q$.
	Thus, every element $\alpha \in \mathbb{F}_{q^2}$
	can be written as $\alpha=x+y\beta$, with $x,y\in \mathbb{F}_q$.

\begin{lemma}\label{l2}
Let $\alpha=x+y\beta \in \mathbb{F}_{q^2}^\ast$, with $x,y \in \mathbb{F}_q$. We have
$f(\alpha)=g(x,y)  \alpha$,
where
\[
g(x,y)= c^2+1+\frac{cy^2}{x^2+xy+y^2a}\in \mathbb{F}_q.
\]
\end{lemma}\label{l5}
\begin{proof}

Since
 $X^2+X+a=(X+\beta)(X+\beta^q)\in \mathbb{F}_{q^2}[X]$, we have $\beta^q=\beta+1$. Thus,
\begin{align*}
f(x+y\beta) & = (x+y\beta) \left( (x+y\beta)^{q^2-q}+c \right) \left( (x+y\beta)^{q-1}+c \right) \\
            & = (x+y\beta) \left( \frac{x+y\beta}{x+y+y\beta}+c \right) \left( \frac{x+y+y\beta}{x+y\beta}+c \right) \\
            & = (x+y\beta) \left( c^2+1+c\left(\frac{x+y\beta}{x+y+y\beta} + \frac{x+y+y\beta}{x+y\beta}\right) \right) \\
            & = \left( c^2+1+\frac{cy^2}{x^2+xy+y^2a} \right) (x+y\beta).
\end{align*}
\end{proof}

\begin{theorem}
\label{fixed2}
We have
\[
|\f (f)|=
\left\{\begin{array}{cl}
   2q-1 & \textrm{if } \operatorname{Tr}(1/c)=1, \\
	1   & \textrm{if } \operatorname{Tr}(1/c)=0.
\end{array}\right.
\]
\end{theorem}
\begin{proof}
Let $x+y\beta \in \f (f)$. Since $0 \in \f (f)$, we consider $x+y\beta\neq0.$ According to Lemma \ref{l2}, it is enough to analyze $g(x,y)=1$. If $y=0$, from Lemma \ref{l2} we have $c=0$, a contradiction. If $y \neq 0$, we write $\lambda =x/y$, and $g(x,y)=1$ is equivalent to  $c^2+1+\frac{c}{\lambda^2+\lambda+a}=1$ or $\lambda^2+\lambda+a+\frac{1}{c}=0.$
By \cite[Proposition 5.9]{livro}, the polynomial
$X^2+X+a+\frac{1}{c} \in \mathbb{F}_q[X]$ is irreducible if and only if $\operatorname{Tr}(a+1/c)=1$, or equivalently $\operatorname{Tr}(1/c)=0$. Therefore, if $\operatorname{Tr}(1/c)=0$, then $0\in \mathbb{F}_{q^2}$ is the only fixed point
of $f$. 

Consider now $\operatorname{Tr}(1/c)=1$.
In this case, there exist $\lambda_1,\lambda_2\in \mathbb{F}_q$ such that
$X^2+X+a+\frac{1}{c}=(X+\lambda_1)(X+\lambda_2)$.  
Therefore, all the fixed points of $f$ different from $0$ are of the form
$y(\lambda_i+\beta)$, where
 $i\in \{1,2\}$ and $y\in \mathbb{F}_q^*$. 
 Since $\lambda_1+ \lambda_2=1$, then
 $\lambda_1 \neq \lambda_2$. Thus,
the $q-1$ points $y(\lambda_1+\beta)$  are different from the points $y(\lambda_2+\beta)$.
Hence, we have a total of $2(q-1)+1=2q-1$ fixed points.
\end{proof}

\begin{example}
Let $q=2^3$. Consider $a \in \mathbb{F}_8$ such that $a^3+a^2+1=0$ and $c =a$. Since $\operatorname{Tr}(a)=1$,
we have that $\{1,\beta\}$ is a basis of $\mathbb{F}_{8^2}$ over $\mathbb{F}_{8}$, where
$\beta^2+\beta +a=0$. The functional graph of $f(X)=X(X^7-a)^9$ is shown in Figure \ref{fig9}.
\end{example}

\begin{figure}[h!]
	\includegraphics[scale=0.8]{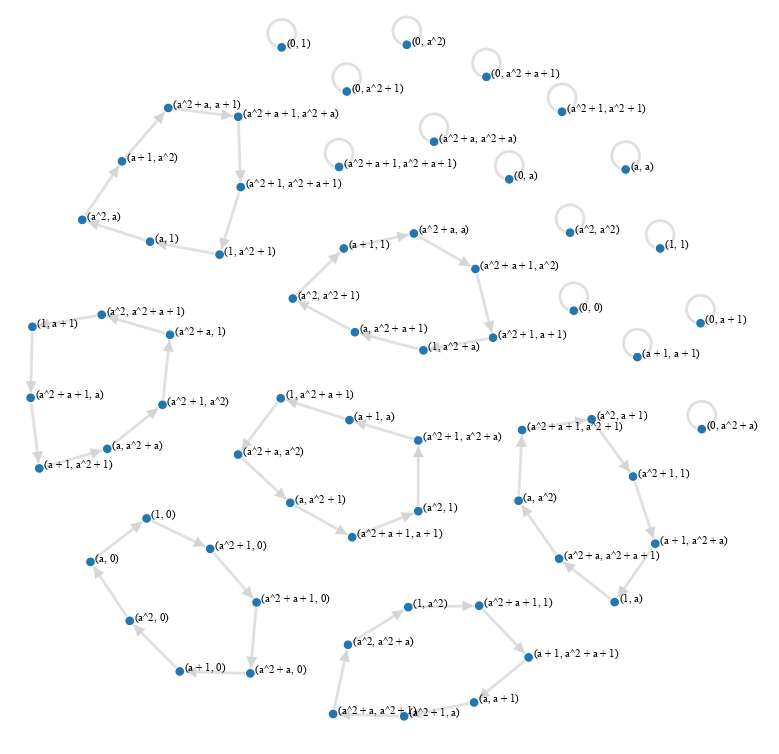}
	\caption{The functional graph of $f(X)=X(X^7-a)^9$ over $\mathbb{F}_{8^2}$}\label{fig9}
\end{figure}


\begin{theorem}\label{ccz2}
We have:
\begin{enumerate}
\item[(a)] The connected component that contains zero is composed only by zero directed to itself if $c\neq1$. 
\item[(b)] The connected component that contains zero is composed by zero directed to itself and the $q-1$ elements
of $\mathbb{F}_q^*$ directed to zero if $c=1$.
\end{enumerate}
\end{theorem}

\begin{proof}
If $c=1,$ then $f(x+y\beta)=0$ is equivalent to $y=0$. From Lemma \ref{l2}, 
	for every point $x+y\beta \in L_{[u:v]}  \backslash \{ 0\}$,  we have $f(x+y\beta) \in L_{[u:v]}  \backslash \{ 0\}$,
	for all sets $L_{[u:v]} \in \mathbb{P}^1$ with $L_{[u:v]} \neq L_{[1:0]}$.
	Therefore, the connected component that contains zero is composed by zero directed to itself and the $q-1$ elements 
	of $\mathbb{F}_q^*$ directed to zero.

Now, suppose $c\neq1$. If $y=0$, then $f(x)=(c^2+1)x$.
Thus, no point of $\mathbb{F}_q^*$ belongs to the connected component that contains zero.
If $y\neq 0$,  
consider $\lambda=\frac{x}{y} \in \mathbb{F}_q$. According to Lemma \ref{l2}, it is enough to analyze $c^2+1+\frac{c}{\lambda^2+\lambda+a}=0$. The last equation is equivalent to $\lambda^2+\lambda+a+\frac{c}{c^2+1}=0.$ By \cite[Proposition 5.9]{livro}, the polynomial 
$X^2+X+a+\frac{c}{c^2+1} \in \mathbb{F}_q[X]$
is irreducible if and only if $\operatorname{Tr}(a+c/(c^2+1))=1.$
Since $X^2+X+\frac{c}{c^2+1}= (X+\frac{1}{c+1})(X+\frac{c}{c+1})$, it  follows that $\operatorname{Tr}(c/(c^2+1))=0.$ Hence, $\operatorname{Tr}(a+c/(c^2+1))=\operatorname{Tr}(a)=1.$ Therefore, $X^2+X+a+\frac{c}{c^2+1} \in \mathbb{F}_q[X]$ is irreducible over $\mathbb{F}_q$ and the connected component that contains zero is composed only by zero directed to itself. 
\end{proof}

\begin{example}
Let $q=2^3$ and $c =1$. Consider $a \in \mathbb{F}_8$ such that $a^3+a^2+1=0$. Since $\operatorname{Tr}(a)=1$,
we have that $\{1,\beta\}$ is a basis of $\mathbb{F}_{8^2}$ over $\mathbb{F}_{8}$, where
$\beta^2+\beta +a=0$. The connected component that contains $0\in\mathbb{F}_{8^2}$ is shown in Figure \ref{fig6}.
\end{example}

\begin{figure}[h!]
	\includegraphics[scale=0.5]{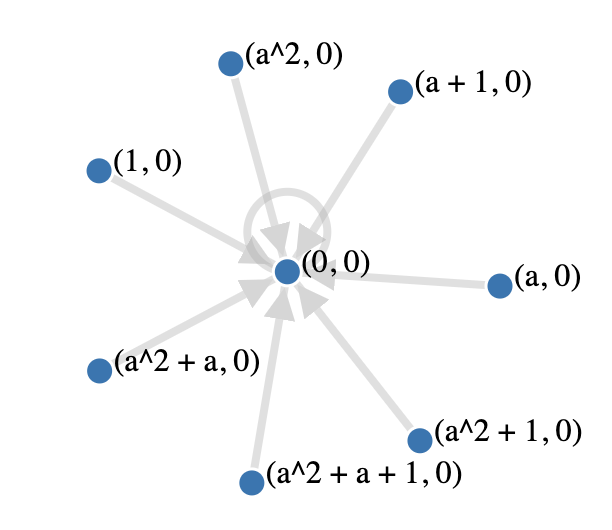}
	\caption{Component of zero for $c=1$ and $q=2^3$}\label{fig6}
\end{figure}

\begin{theorem}\label{t3-v2}
For $L_{[u:v]}  \in \mathbb{P}^1$ we have  $g(u,v) = 0$ if and only if $c=1$ and $v=0$. We also have that 
	for each $L_{[u:v]}  \in \mathbb{P}^1$ such that $g(u,v)\neq 0$, the set
	$L_{[u:v]}\backslash \{ 0\}$ forms $\frac{q-1}{d}$ cycles of length $d$ in the graph $\mathcal{G}(f)$,
	where $d$ is the multiplicative order of $g(u,v)$ in $\mathbb{F}_q^*$.
\end{theorem}
\begin{proof}
The  first part of the statement follows from Theorem \ref{ccz2}.
From Lemma \ref{l2}, we have $f(x+y\beta)=g(x,y)(x+y\beta)$. Note that $g(\lambda x, \lambda y)=g(x,y)$ for all 
$\lambda \in \mathbb{F}_q^*$, thus $g(x,y)$ is constant in each set $L_{[u:v]} \backslash \{ 0 \}$ and
\begin{equation}\label{cycle}
f^{(n)}(x+y\beta)= g(u,v)^n (x+y\beta), 
\end{equation}
for all $n \ge 0$ and
all $x+y\beta \in L_{[u:v]} \backslash \{ 0 \}$. Equation \eqref{cycle} implies that
if $g(u,v) \neq 0$, then each 
set $L_{[u:v]} \backslash \{ 0 \}$ is partitioned into $\frac{q-1}{d}$ cycles of length $d$
in the graph $\mathcal{G}(f)$.
\end{proof}

The following result summarizes Theorem \ref{ccz2} and Theorem \ref{t3-v2}.

\begin{theorem}\label{principaleven}
Let $q$ be an even prime power. The functional graph of the function
$f(X)=X(X^{q-1}-c)^{q+1}$ over $\mathbb{F}_{q^2}$, with $c \in \mathbb{F}_q\backslash \{0,1 \}$, is isomorphic to
\[
		\mathcal{G} = \mathcal{C}_1 \oplus \bigoplus_{L_{[u:v]}\in \mathbb{P}^1} \frac{q-1}{{\ord} (g(u,v))} \times  \mathcal{C}_{{\ord}(g(u,v))}.
\]
If $c=1$, the functional graph is isomorphic to
\[
	\mathcal{G} = (\mathcal{C}_1,\mathcal{T}_{q-1})  \oplus 
		\bigoplus_{L_{[u:v]}\in \mathbb{P}^1 \backslash \{L_{[1:0]} \}} 
		\frac{q-1}{{\ord} (g(u,v))} \times \mathcal{C}_{{\ord}(g(u,v))}.
		\]
	\end{theorem}

\begin{example}
 Let $q=2^4$ and $c =1$. Consider $a \in \mathbb{F}_{16}$ such that $a^4+a^3+a^2+a+1=0$. Since $\operatorname{Tr}(a)=1$,
we have that $\{1,\beta\}$ is a basis of $\mathbb{F}_{16^2}$ over $\mathbb{F}_{16}$, where
$\beta^2+\beta +a=0$. 
The functional graph of $f(X)=X(X^{15}+1)^{17}$ is 
shown in Figure \ref{fig8}.

In Table \ref{table2} we show the number of cycles of the graph
	$\mathcal{G}(f)$ in each set $L_{[u:v]} \backslash \{ 0\}$, where $L_{[u:v]} \neq L_{[1:0]}$.	
	
	{\small
		\begin{table}[h]
			\centering
			\begin{tabular}{ccc}
				$[u:v]$ & $d={\ord }( g(u,v) )$ & Number of cycles of length $d$ \\
				\hline				\hline
$[0: 1], [1:1], [a: 1], [a + 1:1]$ 
                       &  \multirow{ 3}{*}{$5$}         &   \multirow{ 3}{*}{$3+3 +\cdots + 3 = 24$} \\
$[a^2 + a: 1], [a^2 + a + 1: 1]$& & \\
$[a^3: 1], [a^3 + 1: 1]$ & & \\ \hline
$[a^2: 1], [a^2 + 1: 1]$ 
 &  \multirow{ 5}{*}{$15$}         &   \multirow{ 5}{*}{$1+1 +\cdots + 1 = 8$} \\
$[a^3 + a: 1], [a^3 + a + 1: 1]$ & & \\
$[a^3 + a^2: 1], [a^3 + a^2 + 1: 1]$& & \\
$[a^3 + a^2 + a: 1]$ & & \\
$[a^3 + a^2 + a + 1: 1]$ & & \\
 \hline
			\end{tabular}
			\caption{Cycles over $\mathbb{F}_{16^2}$}
			\label{table2}
		\end{table}
	}
From Theorem \ref{fixed2}, it follows that the only fixed point is $0 \in \mathbb{F}_{16^2}$, since
$\operatorname{Tr}(1)=0$.

\end{example}

\begin{figure}[h!]
	\includegraphics[scale=0.8]{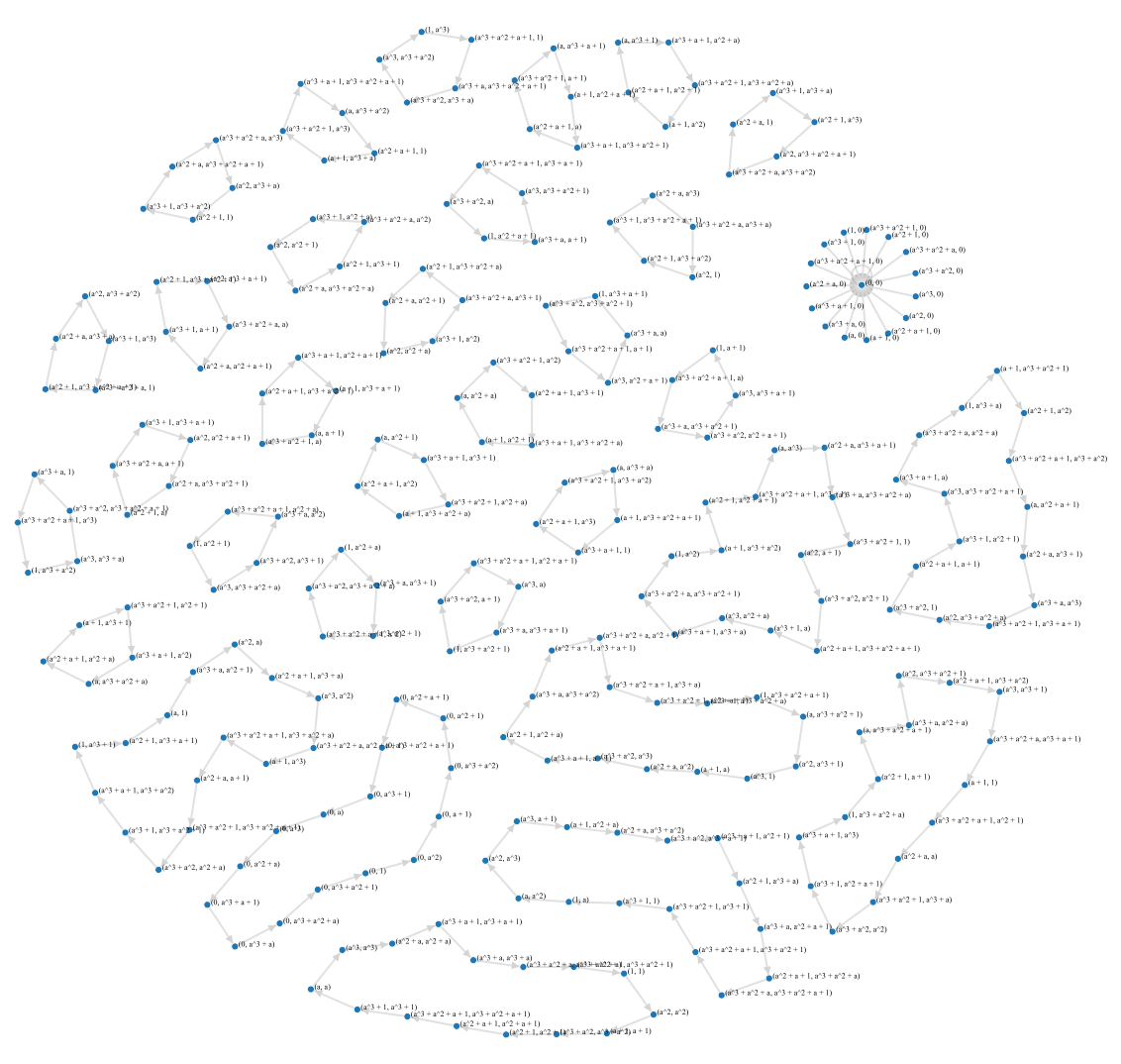}
	\caption{The functional graph of $f(X)=X(X^{15}+1)^{17}$}\label{fig8}
\end{figure}





\section{Existence of maximal cycles}
From Theorem \ref{principalodd} and Theorem \ref{principaleven}, we see that the functional graph of
$f(X)=X(X^{q-1}-c)^{q+1}$ can have cycles of length $d$, where $d$ is a divisor of $q-1$. 
One question that can be asked is for which values of $q$ and $c$, there exist cycles of maximum length $q-1$. 
Observe that this is equivalent to prove that there exists $L_{[u:v]}\in \mathbb{P}^1$ such that
$\ord (g(u,v)) = q-1$. For the odd case, the orders of 
$g(1,0)=(1-c)^2$ and
$g(0,1)=(1+c)^2$ 
depend on $c$ and are at most $\frac{q-1}{2}$. Thus, we can assume that 
$u\neq 0$ and $v\neq 0$.
If we write $\lambda = u/v$, the problem reduces to prove that there exists $\lambda\in \mathbb{F}_q^*$
such that
\[
g(\lambda,1) =\frac{(1-c)^2\lambda^2-(1+c)^2a}{\lambda^2-a}
\]
is primitive.

In the even case, $g(1,0)=1+c^2= (1+c)^2$, 
thus the primitivity of $g(1,0)$ depends on the value of $c$.
In order to solve the odd and even cases simultaneously, we will find all prime powers $q$ for which, for any value of 
$c\in \mathbb{F}_q^*$, there exist $u,v\in \mathbb{F}_q$ with $v \neq 0$ such that $g(u,v)$ is primitive.
Thus, in the even case we also assume that $v \neq 0$, and writing $\lambda = u/v$, the problem reduces to proving that there exists $\lambda\in \mathbb{F}_q$
such that
\[
g(\lambda,1) = c^2+1+\frac{c}{\lambda^2+\lambda+a}
\]
is primitive.

There are some works  in the literature (see \cite{CGNT} and \cite{CSS}) which deal with primitive pairs $(\lambda,g(\lambda))$, where $g$ is a rational functions with coefficients in $\mathbb{F}_q$.
In order to state the results
of this section, we need to recall some definitions.
We say that a rational function
$g \in \mathbb{F}_q(x)$ is {\it not exceptional} if it is not of the form $g(x) = c x^j h^d(x)$, where
$j\in \mathbb{Z}$, $d>1$ divides $q-1$, $c \in \mathbb{F}_q^*$ and $h(x)  \in  \mathbb{F}_q(x)$ (see \cite[Introduction]{CSS}).
For a positive integer $m$, let $W(m)$ denote the number of positive divisors of $m$.

Henceforth, $g(x)$ denotes the polynomial
\[
g(x) =\left\{
\begin{array}{rl}
\frac{(1-c)^2x^2-(1+c)^2a}{x^2-a} \in \mathbb{F}_q(x)        & \textrm{if } q \textrm{ is odd,}\\
c^2+1+\frac{c}{x^2+x+a}                   \in \mathbb{F}_q(x)         &\textrm{if } q \textrm{ is even.}
\end{array}
\right.
\]
In particular, we will use the following results.
\begin{proposition}\label{primitive}
	Let $q$ be a prime power. If $q^{\frac{1}{2}}\ge 3 W(q-1)$, then there exists
	an element $\lambda \in \mathbb{F}_q^*$ such that $g(\lambda)$ is primitive.
\end{proposition}
\begin{proof}
	Observe first that $g(x)$ is not an exceptional function, since its denominator is irreducible.
	Now, if we follow the proof of \cite[Theorem 3.1]{CSS} in detail, but count only the number of elements
	$\lambda \in \mathbb{F}_q$ such that $g(\lambda)$ is primitive, we find that this number is positive if
	$q^{\frac{1}{2}} \ge 3W(q-1)$.
\end{proof}

\begin{proposition}\label{primitive2}
	Let $t>2$ be a positive real number, $q$ a prime power and
	$$
	A_{t}=\prod_{\substack{ s \textrm{ prime} \\ s < 2^t}}
	\frac{2}{\sqrt[t]{s}}.
	$$
	If $q \ge \left(  3 \cdot A_t  \right)^{\frac{2t}{t-2}}$, then there exists $\lambda \in \mathbb{F}_q$ such that $g(\lambda)$ is primitive.
\end{proposition}
\begin{proof}
From the proof of \cite[Proposition 3.6]{CGNT} we have $W(q-1) < A_t \cdot q^{\frac{1}{t}}$ for all positive real number $t>0$.
Thus, if $q^{\frac{1}{2}} \ge 3\cdot A_t \cdot q^{\frac{1}{t}}$ then $q^{\frac{1}{2}} \ge 3W(q-1)$, and from Proposition \ref{primitive} there exists
an element $\lambda \in \mathbb{F}_q^*$ such that $g(\lambda)$ is primitive.
Finally, if $t>2$ then
\[
q^{\frac{1}{2}} \ge 3\cdot A_t \cdot q^{\frac{1}{t}} \Longleftrightarrow 
q \ge \left(  3 \cdot A_t  \right)^{\frac{2t}{t-2}}.
\]
\end{proof}

\begin{theorem}\label{lastresult}
	The functional graph of $f(X)=X(X^{q-1}-c)^{q+1}$ over $\mathbb{F}_{q^2}$ has cycles of maximum length $q-1$
	for all prime power $q$ and all $c \in \mathbb{F}_q^*$. 
\end{theorem}
\begin{proof}
From Proposition \ref{primitive2} and using $t=4.08$ we get that there exists $\lambda \in \mathbb{F}_q$ such that $g(\lambda)$ is
primitive for all prime power $q \ge 44898$. For $q<44898$ we use SageMath to factor $q-1$ and 
conclude that $q^{\frac{1}{2}} \ge 3W(q-1)$ for $q > 8971$. We use SageMath to test the existence 
of an element $\lambda \in \mathbb{F}_q^*$ such that $g(\lambda)$ is primitive, for all prime power
$q \leq 8971$ and all $c \in \mathbb{F}_q^*$. 
\end{proof}

\section*{Acknowledgements}
The authors were partially supported by FAPEMIG grant RED-00133-21.  The second  author was also partially supported by FAPEMIG grant APQ-02546-21 and CAPES, and the third author was also partially supported by FAPEMIG grant APQ-00470-22 and CNPq grant 307924/2023-8.

\end{document}